\documentclass[12pt,fleqn,leqno]{article}

\usepackage[T1]{fontenc} 

\usepackage{textcomp}     

\usepackage[cmex10]{amsmath}
\usepackage{amsfonts}
\usepackage{amssymb}
\usepackage{amsthm}

\usepackage[dvips]{graphics} 
\usepackage{graphicx}        
\usepackage[noend]{algorithm2e}
\usepackage{slashbox}

\theoremstyle{definition} \newtheorem{theorem}{Theorem}[section]
\theoremstyle{definition} \newtheorem{definition}[theorem]{Definition}
\theoremstyle{definition} \newtheorem{lemma}[theorem]{Lemma}
\theoremstyle{definition} \newtheorem{proposition}[theorem]{Proposition}
\theoremstyle{definition} 
\theoremstyle{definition} \newtheorem{observation}[theorem]{Observation}
\theoremstyle{definition} 
\theoremstyle{definition} \newtheorem{remark}[theorem]{Remark}
\title{Central Forests in Trees} 
\date{}
\author{\begin{tabular}[t]{c@{\extracolsep{2em}}c}
     Shrisha Rao\footnote{Corresponding author.}, \ \ Babita Grover \\
    {\tt \{srao,babita\}@iiitb.ac.in} \\
    International Institute of Information Technology \\
    Bangalore 560 100 \\ India
\end{tabular}}

\begin{document}

\maketitle

\begin{abstract}

  A new $2$-parameter family of central structures in trees, called
  central forests, is introduced. Minieka's $m$-center
  problem~\cite{Edward1970} and McMorris's and Reid's
  central-$k$-tree~\cite{McMorrisReid1997} can be seen as special
  cases of central forests in trees.  A central forest is defined as a
  forest $F$ of $m$ subtrees of a tree $T$, where each subtree has $k$
  nodes, which minimizes the maximum distance between nodes not in $F$
  and those in $F$. An $O(n(m+k))$ algorithm to construct such a
  central forest in trees is presented, where $n$ is the number of
  nodes in the tree. The algorithm either returns with a central
  forest, or with the largest $k$ for which a central forest of $m$
  subtrees is possible. Some of the elementary properties of central
  forests are also studied.

\end{abstract}

\noindent{\bf Keywords:}\\

central forest, center, tree, $m$-center, central-$k$-tree, algorithm,
location problems, network center problems

\section{Introduction}

Graph theory has applications in many real life situations. One of the
very interesting applications of graph theory is the \emph{theory of
  facility location} in networks. Graph theorists often focus on
(unweighted) graphs, particularly trees. In the real world, there are
a lot of networks that are organized into trees, such as
ethernet-based campus/enterprise networks, core networks of cellular
networks and telephone networks. These applications are based on the
centrality in trees. There are many kinds of centrality notions like
center of tree, defined by Jordan~\cite{Jordan1869} as a set of nodes
that minimizes the maximum distance to other nodes, centroid of a
tree, again defined by Jordan~\cite{Jordan1869} as the set $x$ of
nodes of tree $T$ that minimizes the maximum order of a component of
$T-x$ and median of tree, defined by Zelinka~\cite{Zelinka1968} as the
set of nodes that minimizes the sum of distances or equivalently the
average distance to other nodes.

Minieka \cite{Edward1970} considered an $m$ center of a tree $T$ as a
set $M$ of $m$ nodes of the tree that minimizes the maximum distance
between every other node of $T$ and $M$. Chandrasekaran and Tamir
\cite{ChandrasekaranTamir1980} presented an algorithm for locating $m$
facilities on a tree network to minimize the maximum distance of the
nodes on the network and their nearest node in $m$-center. This
algorithm takes $O ( ( n\log m )^2 )$ time with $n$ nodes in the
network. Nimrod and Tamir \cite{NimrodTamir1983} presented an
improvement on the previous algorithms for locating $m$ facilities. An
$O(n\log ^3 n)$ algorithm was presented for continuous $m$-center
problem in trees. They also presented an $O(n\log ^2 n\log \log n)$
algorithm for a weighted discrete $m$-center problem.  Steven
et. al.~\cite{Steven1999} gave a self-stabilizing algorithm for
locating centers and medians of trees. Kariv and
Hakimi~\cite{KarivHakimi1979} presented an $O(n \cdot \lg n)$
algorithm for finding the (node or absolute) 1-center; an $O(n)$
algorithm for finding a (node or absolute) dominating set of radius
$r$ and an $O(n^2 \cdot \lg n)$ algorithm for finding a (node or
absolute) $m$-center for any $1 < m < n$ for node-weighted tree. Kariv
and Hakimi \cite{KarivHakimi1979} also proposed an $O(n \cdot \lg ^{m
  - 2} n)$ algorithm for finding an absolute $m$-center (where $3
\leqq m < n$) and an $O(n \cdot \lg ^{m - 1} n)$ algorithm for finding
a node $m$-center (where $2 \leqq m < n$) for a node unweighted tree.

Tamir~\cite{Tamir1988} studied the use of dynamic data structures on
obnoxious center location problems on trees, and on the classical
$m$-center problem on general networks, deriving better complexity
bounds in both the cases.  Again in 1991, Tamir~\cite{Tamir1991}
showed that the continuous $m$-Maximin and $m$-Maxisum dispersion
models are NP-hard for general (nonhomogeneous) graphs. Burkard
et. al.~\cite{Burkard2001} offered an improvement over Tamir's work by
giving a linear-time algorithm for graph that is a path or a star.
Their algorithm is an improvement over Tamir's~\cite{Tamir1988} by a
factor of log($n$) for general trees.

Slater~\cite{Slater1981,JSlater1981,Slater1982} considered the
location facility problem that was path shaped on a tree. Locating
paths with minimum eccentricity and distance, respectively, may be
viewed as multicenter and multimedian problems, respectively, where
the facilities are located on nodes that must constitute a path. A
linear algorithm for finding paths with minimum eccentricity is also
presented. Hedetniemi et. al. ~\cite{Hedetniemi1981}
gave a linear time algorithm for finding the minimal path among all
paths with minimum eccentricity in a tree network. Morgan and
Slater~\cite{Morgan1980} considered the problem of finding a path of
minimum total distance to all other nodes in a tree network with
equal as well as non-equal arc lengths. Minieka~\cite{Minieka1985}
considered central paths and trees with fixed length $L$ in a tree
network.  Slater~\cite{Slater1978,PJSlater1981} studied centers
to centroids and $k$-nucleus of a graph and also Reid~\cite{Reid1991}
considered centroids to centers in trees.

Along with these central sets, $k$-center and $k$-median as given in
Handler and Mirchandani~\cite{Handler1979} and Mirchandani and
Francis~\cite{Mirchandani1990} are of interest in facility location
theory. McMorris and Reid~\cite{McMorrisReid1997} considered a central
$k$-tree is a subtree $C$ of $k$ nodes that minimizes the maximum
distances between nodes not in $C$, and those in $C$.

In this paper, we generalize the work done by
Minieka~\cite{Edward1970} and McMorris and
Reid~\cite{McMorrisReid1997}, by considering a forest $F$ of $m$
subtrees of a tree $T$, where each subtree has $k$ nodes, where $F$
minimizes the maximum distance between nodes not in $F$ and those in
$F$. We introduce this generalization as central forests in
trees. Minieka's~\cite{Edward1970} and McMorris's and
Reid's~\cite{McMorrisReid1997} work can be seen as special cases of
central forest in trees. An $O(n(m+k))$ algorithm to construct such a
central forest in trees is presented, where $n$ is the number of nodes
in the tree. The complete performance analysis and proof of
correctness are given. We study some of the elementary properties of
central forest. We also have the upper bound on the order of subtrees
in central forests as our algorithm either returns with a central
forest, or with the largest $k$ for which a central forest of $m$
subtrees is possible.

Section~\ref{prelim} talks about the notations and the work done by
Minieka~\cite{Edward1970} and McMorris and
Reid~\cite{McMorrisReid1997} in detail.  Section~\ref{CF} defines
formally the central forests in trees. Some of the properties of
central forest and the upper bound on the order of subtrees in central
forests are also discussed in this section. Section~\ref{algos}
presents algorithm to construct central forest in trees along with
performance analysis and proof of correctness. This section also
includes an example demonstrating the working of
algorithm. Section~\ref{conc} concludes the paper and also provides
some pointers for future work.

\section{Preliminaries}   \label{prelim}
\subsection{Notation}       
\begin{definition}                  \label{termsDefi}

  The definitions of some of the terms used throughout the paper are
  as follows:

\begin{itemize}
 
\item[(a)] A \emph{graph} $G$ is an ordered pair \((V(G), E(G))\),
  where \(V(G)\) is a finite nonempty set of \emph{nodes}, and \(E(G)
  \subseteq \{\{v, v'\} \ | \ v, v' \in V(G), v' \neq v\}\) is the set
  of \emph{edges}.  \index{graph}
\item[(b)] The number of nodes in $G$ is called the \emph{order} or
  \emph{size} of $G$.  \index{order}
\item[(c)] Let \(e = \{v, v'\}\) be an edge; we say that $e$ is
  \emph{incident} with $v$ and $v'$, that $v'$ is a \emph{neighbor} of
  $v$, or that $v'$ is adjacent to $v$.  Two edges are \emph{adjacent}
  if they share a common node.  \index{neighbor} \index{adjacent}
\item[(d)] The \emph{degree} of a node $v$, denoted $degree(v)$, is
  the number of edges incident with $v$.  All nodes of degree 1 are
  called \emph{leaves} or \emph{end-nodes}, while other nodes are
  \emph{internal nodes}.  \index{degree of a node} \index{leaves}
  \index{end-nodes} \index{internal nodes}
\item[(e)] A \emph{path} $P$ from $v_1$ to $v_i$ in $G$ is an
  alternating sequence \(P \ \triangleq \ v_1, e_1, v_2, e_2,\)
  \(\ldots, e_{i-1}, v_i\) of nodes and edges such that \(v_1, v_2,
  \ldots, v_i\) are distinct, and for every $j$, \(1 \leq j < i\),
  $e_j$ is an edge incident with nodes $v_j$ and $v_{j+1}$; $i - 1$ is
  the {\it length} of $P$.  If there exists such a path in $G$ such
  that \(\{v_i, v_1\}\) is also an edge of $G$, then $P$ together with
  this edge is a \emph{cycle} of length $i$.  \index{path}
\item[(f)] A graph is \emph{connected} if there is a path between any
  two nodes.  \index{connected}
\item[(g)] The \emph{distance} between nodes \(v_i\) and \(v_j\),
  denoted \(d(v_i, v_j)\), is the length of a shortest path between
  them.  \index{distance} \index{d($v_i, v_j$)}
\item[(h)] The \emph{degree} of a graph $G$, denoted \(\Delta (G)\),
  is the maximum degree of any node in $G$.  \index{degree of a graph}
\item[(i)] If $G$ and $G'$ are graphs, $G'$ is a \emph{subgraph} of $G$ if \(V(G') \subseteq V(G)\) and \(E(G') \subseteq E(G)\).  
\index{subgraph}
\item[(j)] A \emph{connected component} of a graph $G$ is a maximal
  connected subgraph of $G$.  \index{connected component}
\item[(k)] A graph that is \emph{connected} and has no cycles is
  called a {\it tree}.  \index{tree}
\item[(l)] The \emph{eccentricity} of a node $x$ in a connected graph
  $G$, denoted \(e(x)\), is given by:
\[ e(x) \ \triangleq \ \mathrm{max}\{d(x, y) \ | \ y \in V(G)\}.\]
\index{eccentricity}
\item[(n)] If \(V \subset V(T)\), the \emph{subtree induced by} $V$ can be defined in two ways:
\begin{itemize}
\item[(i)] The subgraph of $T$ induced by $V$ is the forest
  with node set $V$ and edge set given by all edges of $T$ with both
  ends in $V$.
\item[(ii)] The subgraph of $T$ induced by $V$ is the smallest
  subtree of $T$ containing all nodes of $V$.
\end{itemize}
\index{subtree}
\item[(o)] A \emph{forest} in a graph is a disjoint union of trees. In
  trees, a forest is a disjoint union of subtrees.  \index{forest}
\end{itemize}

\end{definition}
\subsection{Center}
\index{center}
Center of a graph $G$, denoted by center($G$) is defined as follows: 
\begin{definition}
If $c \subseteq V(G)$ and $|c| = 1$ or $2$, then 
\[ center(G) \ \triangleq \ \{c \ | \ c \in V(G), e(c) \leq e(y), \ \forall y \in V(G)\}. \]
\end{definition}

In other words, the center of a graph is a set of nodes that minimizes
the maximum distance to other nodes of the graph. For trees,
Jordan~\cite{Jordan1869} has given a pruning algorithm to find the
center. The algorithm works in steps. In each step, the end (leaf)
nodes of the tree and their connecting edges are removed to get a new
tree. When we keep on pruning the tree like this, we are left with
either a single node or two nodes joined by an edge. This set of $1$
or $2$ nodes is the unique center of the tree. This set of nodes has
the minimum eccentricity. In the Figure \ref{centr}, the dashed edges
show the nodes pruned at the first step and the dotted edges denote
the nodes pruned at the second step of the algorithm. At the third
step only a single node is left, and that is the unique center of the
tree. For clarity, the central node of tree is indicated by an oval
around it.

\begin{figure}[htp]
\begin{center}
\includegraphics [scale=.35, angle = -90] {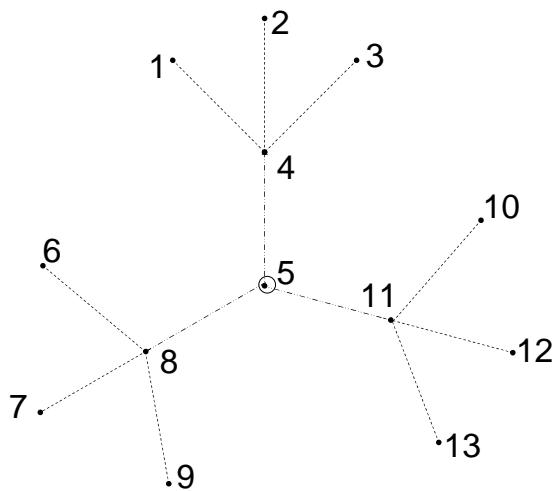}
\caption{An example of a center of a tree}
\label{centr}
\end{center}
\end{figure} 
    
\subsection{$m$-Center} 

Minieka~\cite{Edward1970} has generalized the concept of a central
node of a tree $T$ as follows:

\begin{definition}  \label{mDefi}
If \(X \subseteq V(T)\) and \(u \in V(T)\), define
\[ d(u, X) \ \triangleq \ \mathrm{min} \ d(u, x_i), \ x_i \in X \]
\[ e(X) \ \triangleq \ \mathrm{max} \ d(u, X), \ u \in V(T) \]
\[ center_m (T) \ \triangleq \ \{X \ | \ e(X) \leq e(Y), \ \forall \, Y \subseteq V(T), |Y| = m\} \] 
\end{definition}

Minieka~\cite{Edward1970} defines the $m$-center by ``an $m$-center
set of a graph is any set of $m$ nodes, belonging to either the nodes
or edges, that minimizes the maximum distance from a node to its
nearest node in $m$-center.'' In other words, the problem of
\index{$m$-center}$m$-center for $m$ = $1,2$..., is to find the set of
nodes $X$ that minimizes the maximum distance between a node of $G$
and its nearest node in $m$-center.

The members of set $X$ are called as $m$-center of tree $T$. The set
$X$ need not be unique. Minieka has also given a method for solving
the $m$-center problem by solving a finite series of minimum set
covering problems. An example of an $m$-center is as shown in Figure
\ref{mcenter} for $m=3$. In this figure nodes $4,8$ and $11$ makes up
the $3$-center of the tree, i.e. here $X$ = \{$4,8,11$\}.
\begin{figure}[htp]
\begin{center}
\includegraphics [scale=.35, angle = -90] {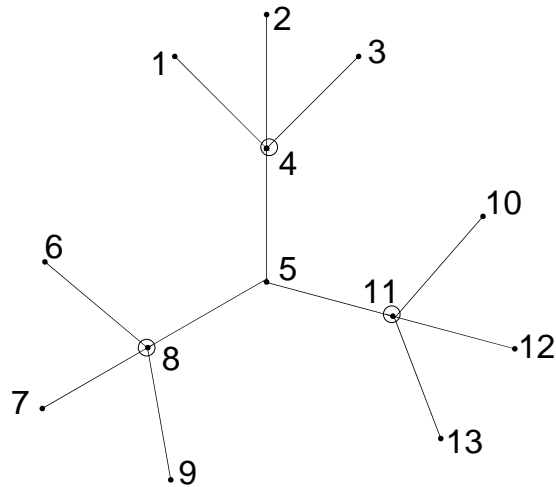}
\caption{An example of a $3$-center of a tree}
\label{mcenter}
\end{center}
\end{figure} 

NB: In the literature, an $m$-center of the tree (graph) is also
sometimes known as the $p$-center of the tree (graph).  But, as
Minieka originally proposed it as $m$-center, we are using the term
$m$-center instead of $p$-center.

\begin{remark}\label{pigeonhole}

We need to impose the further restriction that \(\Delta T \leq m\).

\end{remark}

The justification for this is that if the restriction is not observed,
then for some trees, e.g., a star of degree \(\Delta(T)\), the
$m$-center may not be properly defined.  \(\Delta T\) nodes in a star
can be properly placed as nodes of the $m$-center, but if \(m >
\Delta T\), then by the pigeonhole principle, at least one limb of the
star must have more than one node belonging to the $m$-center, which
is not sensible.

\subsection{Central $k$-Trees}       \label{centralktree}

Central $k$-trees were introduced by McMorris and Reid in
1997~\cite{McMorrisReid1997}.

\begin{definition} \label{cktDefi} 

  If the tree $T$ be of order $n$ and $W_{T,k}$ be the set of all the
  subtrees of $T$ of order $k$, then a central-$k$-tree is defined as
\[ \mathrm{Central}$--$k$--$\mathrm{tree}(T) \ \triangleq \ \{ W \ | \ e(W) \leq e(W'), \ \forall W' \in W_{T,k}, k\leq n\}. \]

\end{definition}

\begin{figure}[htp]
\begin{center}
\includegraphics [scale=.35, angle = -90] {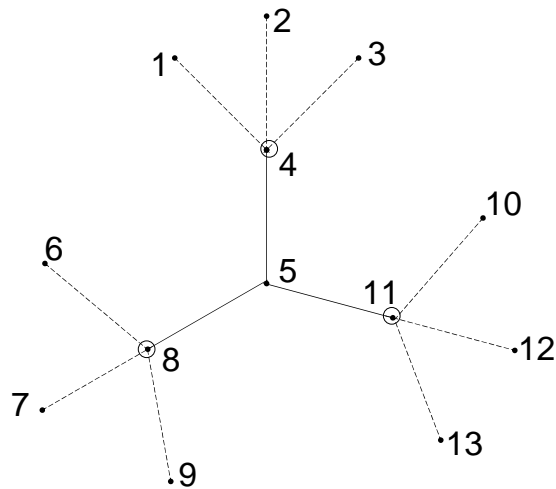}
\caption{An example of a central-$4$-tree of a tree}
\label{CKT}
\end{center}
\end{figure} 

An example of central-$k$-tree is shown in Figure \ref{CKT} for
$k=4$. The algorithm for central-$k$-tree for tree $T$, from their
paper is given as algorithm \textsc{Central-$k$-tree}. The tree $T$ of
Figure \ref{mcenter}, when pruned as per algorithm
\textsc{Central-$k$-tree}, yields a central-$4$-tree as depicted in
Figure \ref{CKT}. The pruned edges are shown as dotted lines.

The pruning process is derived from the procedure used to prove that
the tree center consists of a single node or two adjacent nodes. In
fact, if $T$ has a single node in its center and $k=1$, then this
pruning process gives a unique center. And if $T$ has two nodes in its
center and $k=2$, then also this process yields a unique center
consisting of two adjacent nodes.

The \textsc{Central-$k$-tree} algorithm prune off all end-nodes
repeatedly, as in the pruning process for determining the center,
until a subtree $T'$ is obtained, where \(|V(T')| \leq k\). If
\(|V(T')| = k\), then $T'$ is a central $k$-tree; if not, add vertices
and incident edges to $T'$ until \(|V(T')| = k\).  The result then is
a central-$k$-tree.

In Algorithm \ref{centktree} (\textsc{Central-$k$-tree}), $T(i)$ and
$L(i)$ are two sets of nodes, these sets contain un-pruned nodes and
pruned nodes at each step of algorithm, respectively. $Z$ is a set
that holds the subset of $L(i)$. The set $U$ gives the
central-$k$-tree at the end of the algorithm.

\incmargin{1em}
\restylealgo{boxruled}
\linesnumbered
\begin{algorithm}[H]               \index{Central-$k$-tree Algorithm}      
\label{centktree}
\dontprintsemicolon
$T(0) \leftarrow T$ \;
$L(0) \leftarrow$ \{$v \in V(T)|v$ is an end-node of $T$\} \;
$i \leftarrow 1$ \;
$T(i) \leftarrow T(i-1)/L(i-1)$ \;
\If{$|V(T(i))| \leq k$}{
$U \leftarrow V(T(i)) \cup Z$, where $Z$ is any $(k-|V(T(i))|)$-subset of $L(i-1)$ \;
Go to line 12 \;
}
$L(i) \leftarrow$ \{$v$ $\in V(T(i))|v$ an end-node of $T(i)$\} \;
$i \leftarrow i+1$ \;
Go to line 4\;
Stop \;
\caption{\textsc{Central-$k$-tree}($T$, $k$)}
\end{algorithm}
\decmargin{1em}

In Algorithm \ref{centktree}, $L(0)$ is set as the end nodes of
$T$. $T(i)=T(i-1)/L(i-1)$ means that the $T(i)$ is set to the nodes
those are in $T(i-1)$ but not in $L(i-1)$, the corresponding edges are
also not there. At line $5$, nodes in $T(i)$ are compared with $k$, if
$|V(T(i)| \leq k$, then a subset nodes from $L(i-1)$ is added to $U$
to make number of nodes in $U$ exactly $k$. And the algorithm
exits. But if the condition at line $5$ fails, $L(i)$ is calculated
again and $i$ is incremented and control goes back to line number $4$.

The working of algorithm in detail can be shown using example tree $T$
of Figure \ref{mcenter}. Let us suppose that we want to get
central-$3$-tree. At line $1$, $T(0)$ is set to
$\{1,2,3,4,5,6,7,8,9,11,12,13\}$. Line $2$ sets $L(0) =
\{1,2,3,6,7,9,10,12,$ $13\}$ because these are the end nodes
(leaves). At line $4$, $T(1) = \{4,5,8,11\}$ as shown in Figure
\ref{CKT}. The `if' condition at line $5$ fails, so at line $9$,
$L(1)$ is set to $\{4,8,11\}$. At line $10$, $i$ is incremented and
control goes back to line $4$.

Here, $T(2)$ is set to $\{5\}$. This time the `if' condition is
$\mathbf{TRUE}$, so a subset $Z$ of cardinality $2$ ($3-1$) is chosen
from $L(1) = \{4,8,11\}$. Let the chosen set $Z$ is $\{4,8\}$. $U$ is
set to $\{5\}\cup \{4,8 \}$, i.e $U = \{5,4,8\}$ and algorithms
stops. The set $U$ gives us the central-$3$-tree depicted in Figure
\ref{Central3tree} of tree shown in Figure \ref{mcenter}. The pruned
edges are shown as dotted lines.

\begin{figure}[htp]
\begin{center}
\includegraphics [scale=.35, angle = -90] {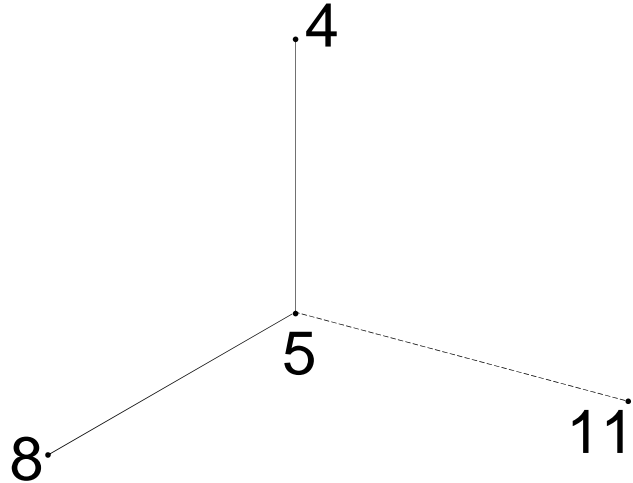}
\caption{An example of a central-$3$-tree of a tree}
\label{Central3tree}
\end{center}
\end{figure} 

If a different subset $Z$ is chosen at line $6$, we may get a
different central-$3$-tree. This shows that a central-$k$-tree is not
unique.

There might be central $k$-trees that do not arise from the algorithm
as per McMorris's and Reid's ~\cite{McMorrisReid1997} remark:

\begin{quote}
  There might be central $k$-trees that do not arise from the
  algorithm. For example, if $k=3$ and $T$ is the $7$-tree obtained by
  subdividing each edge of the complete bipartite graph $K(1,3)$, then
  each of three distinct subtrees of $T$ isomorphic to $K(1,2)$ are
  central 3-trees of $T$ arising from the algorithm. However, each of
  the three subtrees of order of order $3$ consisting of $2$-path
  starting from an end-node of $T$ is a central $3$-tree of $T$ as
  well. Line $6$ in the algorithm could be altered to allow $Z$ to be
  any $(k-|V(T(i)|)$-subset of $V(T)$ so that $T[U]$ is a $k$-tree,
  and all central $k$-trees would be produced by all such choices. The
  restriction of $Z$ to nodes in $L(i-1)$ in line $6$ insures that
  $T[U]$ is a subtree.
\end{quote}

McMorris and Reid also give a proposition which states that for any
integers $n$ and $k$, where $n \geq k \geq 1$, every $k$-tree is a
central $k$-tree for some $n$-tree.

\section{Central Forests}   \label{CF}

Let $k$ be the order of the subtrees of the tree $T$, and
$\mathbb{F}_{T,m}$ the set of all forests in $T$ of $m$ subtrees each.
\begin{definition}  \label{CFDefi}
The eccentricity of a single forest \(F \in \mathbb{F}_{T,m} \) is
  given by:
\[ e_m (F) \ \triangleq \ \mathrm{max} \ d(x, F), \forall x \in V(T). \] 
Then the central forests of $T$ can be given by:
\[ C(T; m, k) \ \triangleq \ \{ F \ | \ e_m(F) \leq e_m(F'), \ \forall F' \in \mathbb{F}_{T,m} \}. \]
\end{definition}

A single forest $F$ is a set of nodes, these nodes are divided into
$m$ subtrees each of order $k$. The $e_m(F)$ calculates the
eccentricity of forest $F$, i.e., the maximum distance between a node
in $T-F$ and its nearest node in $F$.

Here, $C(T;m,k)$ denotes the central forest of $m$ subtrees of tree
$T$ with each subtree being of order $k$. The central forest is
defined as the forest $F$ in $\mathbb{F}_{T,m}$ such that $F$ has
minimum of maximum distance between a node in $T-F$ and its nearest
node in $F$, among all the forests $F'$ in $\mathbb{F}_{T,m}$ i.e
$e_m(F) \leq e_m(F')$. Of course, $F$ need not be unique.

\begin{remark} \label{remark1}

  Note that for the definition of \(C(T; m, k)\) to be meaningful, we
  need to impose the restriction:
\[ mk \leq |V(T)|. \]

\end{remark}

The reason for this is that in a central forest, there are $m$
subtrees and order of each subtree is $k$. Hence, the total number of
nodes in central forest is $m\times k$. Obviously, the number of nodes
in central forest cannot exceed the number of nodes in $T$. So \( mk
\leq |V(T)|. \)

Figure~\ref{CMTK} gives an example of a tree with a central forest,
specifically, \(C(T; 2, 3)\). There are $2$ subtrees of order $3$ that
make the central forest, i.e. $m = 2$ and $k = 3$. The nodes $3,4,5$
and $12,13,14$ make up the two subtrees which minimizes the maximum
distance between a node in $\{1,2,6,7,8,9,10,11,15,16\}$ and its
nearest node in $\{3,4,5,12,13,14\}$. The subtrees are surrounded by
ovals for clarity.

\begin{figure}[htp]
\begin{center}
\includegraphics [scale=.25, angle = -90] {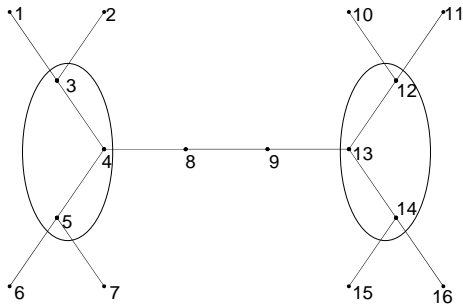}
\caption{An example of a central forest.}
\label{CMTK}
\end{center}
\end{figure} 

\begin{proposition} \label{spcasemcenters} 

  The $m$-center is a special case of a central forest $C(T;m,k)$ for
  tree $T$, i.e., a $C(T;m,1)$ is equivalent to an $m$-center of tree
  $T$.

\end{proposition}

\begin{proof}

 As per Definition \ref{CFDefi},
 \[ C(T; m, k) \ \triangleq \ \{ F \ | \ e_m(F) \leq e_m(F'), \
 \forall F' \in \mathbb{F}_{T,m} \}. \] $F$ is a collection of $m$
 subtrees of $T$. If $F$ is a central forest, then the eccentricity of
 subtrees of $F$ is minimum among all other possible sets of
 $m$-subtrees of tree $T$.

 If the order of each subtree is $1$, then the set $X$ of these $m$
 nodes makes the central forest $F$ and we have
 \[ C(T; m,k) \ \triangleq \ \{X \ | \ e(X) \leq e(Y), \ \forall \, Y
 \subseteq V(T), |Y| = m\} \] By Definition \ref{mDefi}, this is the
 $m$-center of tree $T$. \qedhere

\end{proof}


\begin{proposition}

  The central-$k$-tree is a special case of central forest $C(T;m,k)$
  for tree $T$, i.e. $C(T;1,k)$ is equivalent to the central-$k$-tree
  of tree $T$.

\end{proposition}

\begin{proof}

As per Definition~\ref{CFDefi},  
\[ C(T; m, k) \ \triangleq \ \{ F \ | \ e_m(F) \leq e_m(F'), \ \forall
F' \in \mathbb{F}_{T,m} \}. \] The central forest $F$ is a collection
of $m$ subtrees of $T$, each subtree of order $k$. If $F$ is a central
forest, then the eccentricity of subtrees of $F$ is minimum among all
other possible sets of $m$-subtrees of tree $T$.

If there is only one subtree $W$ in $F$ of order $k$, then this
subtree makes up the central forest and we have
\[ C(T; m, k) \ \triangleq \ \{ W \ | \ e_m(W) \leq e_m(W'), \ \forall
W' \in \mathbb{F}_{T,m} \}. \] By Definition \ref{cktDefi}, this is a
central-$k$-tree in tree $T$.

\end{proof}


\begin{observation}

  The definition of a central forest is a
  \index{generalization}generalization of the $m$-center and the
  central-$k$-tree.


\end{observation}

\subsection{Types of Subtrees in the Central Forest}

There is only one type of tree of order 2, and only one of order 3, so
\(C(T; m, 2)\) and \(C(T; m, 3)\) do not split into cases implied by
the types of trees of those orders.  In general, however, it is
possible to have multiple sub-cases for \(C(T; m, k)\).  For instance,
since there are two types of trees of order 4 (Figure~\ref{size4}),
there are \emph{three} possible cases for \(C(T; m, 4)\) (those
containing subtrees of the first kind, those containing subtrees of
the second kind, and those containing both kinds).  In general, \(C(T;
m, k)\) is the union of the various sub-cases, some of which may be
empty.  The number of nonisomorphic trees of order $k = 1, 2, 3,
\ldots$ are $1, 1, 1, 2, 3, 6, 11, 23, 47, 106, 235, 551, 1301,$
$3159, \ldots$. The generating function for this sequence is
\[ A(x) = 1 + T(x)-\frac{T^2(x)}{2}+\frac{T(x^2)}{2} \]
where \[ T(x) = x + x^2 + 2x^3 + 4x^4 + \ldots \] satisfies 
\[ T(x) = x \exp(T(x)+\frac{T(x^2)}{2}+\frac{T(x^3)}{3}+\frac{T(x^4)}{4}+ \ldots)\] as shown by Sloane \cite{Sloane} and the references therein.

\begin{remark}
  If $x$ is the number of nonisomorphic trees of order $k$, then the
  maximum number of possible combinations of different types of
  subtrees for central forest $C(T;m,k)$ is $x^m$.
\end{remark}

Each subtree can be of one of the $x$ types of nonisomorphic
trees. There are total $m$ subtrees in a central forest. So the
maximum number of combinations of different types of subtrees for a
central forest $C(T;m,k)$ is $x^m$.

\begin{figure}[htp]     
\centering
\includegraphics [scale=.35, angle=-90]{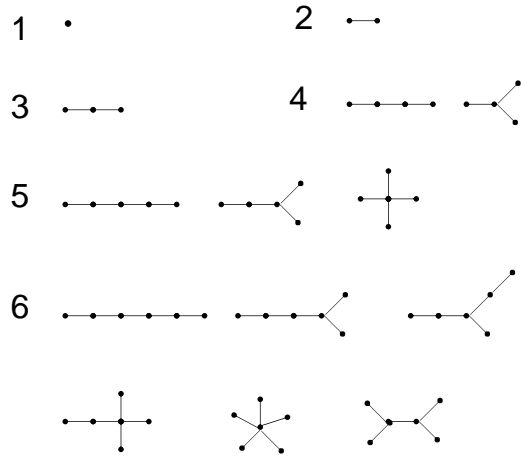}
\caption{Nonisomorphic trees of order $1,2,3,4,5,6$}
\label{size4}
\end{figure} 



\subsection{Properties of Central Forests in Trees}
\begin{observation}
  The $1$-centers of the members of $C(T; m, k)$, do not necessarily
  consist of $m$-center of $T$,\( \forall m, 1 \leq m \leq \lfloor
  \frac{|V(T)|}{k}\rfloor\).
\end{observation}

We can show this with the help of an example of a $C(T;m,k)$. Consider
a tree $T$ as shown in Figure \ref{mextree}, with small ovals around
the nodes $2$ and $5$ indicating that these two nodes are $2$-centers
of the tree $T$. In this figure the big ovals around the nodes
represent the $2$ subtrees of order $4$ for $C(T;2,4)$. If we prune
these subtrees, we get the node $6$ as one of the $1$-center, which is
not originally a node in the $m$-center of the tree $T$. Therefore,
the example clearly shows that $1$-centers of the members of $C(T; m,
k)$, do not necessarily consist of $m$-center of $T$.

\begin{figure}[htp]     
\centering
\includegraphics [scale=.25, angle=-90]{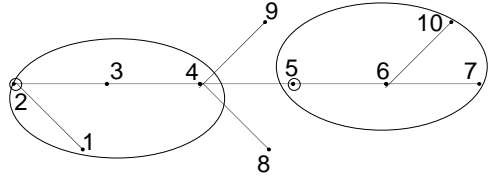}
\caption{A central forest $C(T;2,4)$}
\label{mextree}
\end{figure} 

\begin{theorem}      \label{contains}
Every \(C(T;m,k)\) contains a \(C(T;m,k-1), \
\forall m, \ 1 \leq m \leq \lfloor \frac{|V(T)|}{k}\rfloor\).
\end{theorem}
\begin{proof}

  If we prune any tree of order $n$ to a tree of order $k$, we get a
  central-$k$-tree $W$, and by Definition \ref{cktDefi} we have
\[ \mathrm{Central}$--$k$--$\mathrm{tree}(T) \ \triangleq \ \{ W \ | \ e(W) \leq e(W'), \ \forall W' \in W_{T,k}, k\leq n\}. \]

As in Definition \ref{cktDefi}, $W_{T,k}$ is all subtrees of $T$ of
order $k$.  Now, if we prune $W$ to a tree $W'$ of order $k-1$, the
eccentricity of $W'$ will also be minimum.

Thus, if we prune each member of $C(T;m,k)$ to subtrees of order
$k-1$, we get all the members with minimum eccentricity. That is
\[ C(T; m,k-1) \ \triangleq \ \{ F \ | \ e_m(F) \leq e_m(F'), \ \forall F' \in \mathbb{F}_{T,m} \}. \]

This above equation is the definition of $C(T;m,k-1)$. Hence, every
\(C(T; m, k)\) contains a \(C(T;m,k-1)\).  \qedhere

\end{proof}

\begin{theorem} \label{mcenterpart} 

  All nodes of any $m$-center are part of members of some $C(T;m,k)$
  of tree $T$, \( \forall m, 1 \leq m \leq \lfloor
  \frac{|V(T)|}{k}\rfloor\).

\end{theorem}

\begin{proof}

  We prove this theorem by induction on $k$.

  Base case: For $k=1$, the $m$ subtrees have only one node each. Then
  by Proposition \ref{spcasemcenters}, the members of $C(T;m,1)$
  contain an $m$-center.

  Inductive step: Let members of $C(T;m,k-1)$ contain an $m$-center of
  tree $T$.  By this assumption, the $m$ members of $C(T;m,k-1)$
  contain $m$-center, and by Theorem \ref{contains}, every $C(T;m,k)$
  contains a $C(T;m,k-1)$. Therefore, the members of $C(T;m,k)$ also
  contain $m$-center of tree $T$. \qedhere

\end{proof}

\subsection{The Upper Bound on the Order of Subtrees in a Central Forest}  \label{upbound}

By Remark~\ref{remark1}, we have a bound $mk$ $\leq$ $|V|$ that
implies $k$ $\leq$ $\frac{|V|} {m}$. But this is a very loose bound as
the bound on $k$ depend on the topology of the tree $T$.


\begin{figure}[htp]
\begin{center}
\includegraphics [scale=.25, angle = -90] {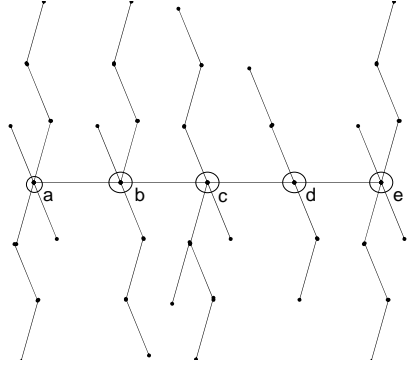}
\caption{A tree with adjacent nodes forming the $m$-center.}
\label{adjacentm}
\end{center}
\end{figure} 


\begin{remark}   \label{adjm}
  When all the nodes of the $m$-center are adjacent to each other, the
  upper bound on $k$ can be obtained by pruning the edges connecting
  the nodes of $m$-center. This pruning of edges give us $m$
  subtrees. In central forest, each subtree should be of same
  order. Therefore, the subtree with minimum number of nodes decide
  the order for all subtrees in central forest. And that is the upper
  bound on $k$.
\end{remark}

An example tree where the nodes of the $m$-center are adjacent to each
other is shown in Figure \ref{adjacentm}. In this figure the
$5$-center is $\{a, b, c, d, e\}$, as shown by small ovals around the
nodes. By Remark \ref{adjm}, if we prune the edges connecting this
$5$-center, $d$ has minimum number of nodes ($4$) in its subtree. So
the upper bound on $k$ is $4$.

But when the nodes of the $m$-center are not adjacent, finding the
upper bound on $k$ is very difficult as it depends on the topology of
the tree. In Section \ref{algos}, we give an algorithm for
constructing a central forest. This algorithm gives us the upper bound
on $k$ when a central forest of $m$ subtrees, for the required value
of $k$ is not possible.

\section{An Algorithm \textsc{CF} to Construct a Central Forest} \label{algos}

Given the topology of the tree $T$, we present an algorithm
\textsc{CF} to construct the central forest $C(T;m,k)$ as defined in
Section~\ref{CF}. The strategy followed by this algorithm is to divide
the nodes of $T$ into $m$ subtrees such that the nodes are assigned to
their nearest node in
$m$-center.

The number of nodes in these subtrees may differ as per the topology
of the tree $T$. There are two cases to consider, first when $k$ (the
order of each subtree in the central forest) is less than or equal to
the number of nodes in each of the $m$ subtrees. Second, when for one
or more subtrees, $k$ is greater than the number of nodes present in
those subtrees.

In the first case, we simply use the pruning algorithm given by
McMorris and Reid ~\cite{McMorrisReid1997} to get a
central-$k$-subtree for all the $m$ subtrees. These
central-$k$-subtrees for all subtrees give us the central forest
$C(T;m,k)$ in tree
$T$.

In the second case, we try to extend all the subtrees with less nodes
than $k$, by taking nodes from the neighboring subtrees. The neighbor
subtree can itself take nodes from its neighbors, and so on. In this
way, the second case is converted to the first. While taking nodes
from other subtrees, the nodes that give the minimum increases in
eccentricity of all subtrees are chosen. But if one or more subtrees
cannot be extended to contain $k$ nodes, the \textsc{CF} algorithm
outputs the minimum of number of nodes in a subtree among all
subtrees, as the maximum order of the subtrees for which a central
forest is possible.

Algorithm \ref{ST} builds the $m$ subtrees around the $m$-center by
assigning each node to its nearest node in $m$-center. Algorithm
\ref{extendST} is used to extend one or more subtrees, if
required. Pruning of subtrees, to get subtrees of order $k$, is done
using Algorithm \ref{pruneST}. The Algorithm \ref{algoCF} is the main
algorithm which uses above said algorithms and outputs either the
central forest or the maximum value of $k$ for which central forest of
$m$ subtrees is possible for tree $T$.

\subsection{Notation Used in Algorithms}   \label{not}
\begin{itemize}

\item The set \index{$V(T)$ set}$V(T)$ is the set of nodes in tree
  $T$. The number of subtrees in the central forest is denoted by
  \index{$m$}$m$. The order of each subtree in the central forest is
  \index{$k$}$k$.

\item The set $V_m$ holds all the nodes that make the $m$-center. In
  $V_m$ all the nodes of the $m$-center are indexed as per their
  occurrence.  \index{$V_m$ set}

\item The two-dimensional matrix \index{$ST$ matrix}$ST$ contains all
  the nodes in the tree $T$ which are not a part of the set $V_m$. All
  nodes are assigned to the nearest node in the $m$-center. Each row
  of matrix $ST$ holds one of the nodes in $m$-center and all the
  nodes assigned to that node of $m$-center. The nodes of the
  $m$-center in $ST$ matrix appear in the order as in set $V_m$. The
  number of nodes in each row of $ST$ can be different, as it depends
  on how many nodes are assigned to each node of the $m$-center. If
  two or more nodes of the $m$-center are at equal distance from a
  node, then that node is made to wait till end and then it is
  assigned to the node of $m$-center which has lesser number of nodes
  in its corresponding row in matrix $ST$. If two or more nodes of the
  $m$-center have equal number of nodes in their corresponding row in
  the matrix $ST$, then the node is assigned to the node of $m$-center
  with lower value of index in set $V_m$.

\item The one-dimensional matrix \index{$H$ matrix}$H$ holds all nodes
  except the nodes of $m$-center.

\item The two-dimensional matrix $u[r]$ holds the indices of all the
  nodes of the $m$-center that are at equal distance from node $r$.

\item The one-dimensional matrix \index{$z[l]$ matrix}$z[l]$ holds the
  number of nodes in the $l^{th}$ row of $ST$, i.e. $ST[l]$.

\item \index{$x[r]$ matrix}The $x[r]$ is a two dimensional matrix that
  stores the nodes that has to be removed from $ST[r]$ in order to
  remove $v$ as they are reachable via $v$ only, for nodes in
  $ST[r]$. $|x[r]|$ is the number of nodes in $x[r]$ . If not set,
  $|x[r]|$ holds value $0$.


\item The one-dimensional matrix \index{$d[r]$ matrix}$d[r]$ holds the
  number of nodes left in $ST[r]$ when $|x[r]|$ nodes are taken out.

\item The one-dimensional matrix \index{$n[r]$}$n[r]$ holds the number
  of nodes returned by \textsc{ExtendST} algorithm.

\item The variable $color[v]$ holds either color `$White$' or
  `$Gray$'. Initially all the nodes have `$White$' color.
  \index{$color[v]$}

\item The three dimensional matrix $B$, is used to store all the $ST$s
  that is produced by the algorithm.  \index{matrix $B$}

\item \index{$C$ matrix} $C$ is a two-dimensional matrix which stores
  one row of $ST$ matrix at a time and keeps changing as the algorithm
  proceeds.

\item \index{$L$ matrix} $L$ is a two-dimensional matrix which holds
  the end nodes of $C$ at various stages.

\item \index{$X$} The set $X$ is any subset of $L(i)$ which has
  $k-|V(C)|$ nodes, where $i$ is the index of the row of the $L$
  matrix.
  
\item \index{$U$ matrix} $U$ is a two-dimensional matrix, each row of
  which holds subtrees of order $k$, that collectively form the
  central forest.


\item The two-dimensional matrix $A$ contains $m$ rows and $k$
  columns. Each row of this matrix represents one of the $m$ subtrees
  of order $k$. In other words, $A$ holds the central forest
  $C(T;m,k)$ of $T$.

\index{$A$ matrix}
\item \index{function!\textsc{ExtractMin}($r,V_m$)}The \textsc{ExtractMin}($r,V_m$) function returns the distance between the node $r$ and the nearest node in $V_m$.
 

\item The \index{function!\textsc{ExtractMinNum}($r,u[r]$)}function
  \textsc{ExtractMinNum}($r,u[r]$) returns the minimum number of nodes
  among the rows of $ST$ belonging to row $u[r]$.

\item \index{function!\textsc{ExtractMinNumRow}($r,u[r]$)}The function
  \textsc{ExtractMinNumRow}($r,u[r]$) returns the $ST$ row that has
  minimum number of nodes among the rows of $ST$ belonging to row
  $u[r]$. If two or more rows have minimum number of nodes, then the
  row with lower index value is returned.

\item \index{function!\textsc{NodesToBeRemoved}($r,v$)}The
  \textsc{NodesToBeRemoved}($r,v$) is a function that returns the
  nodes that are reachable for any node in $ST[r]$, only through $v$,
  including $v$.

\item \index{function!\textsc{AddRemoveNodes}($x[r]$, $ST[l]$,
    $ST[r]$)}The function \textsc{AddRemoveNodes}($x[r]$, $ST[l]$,
  $ST[r]$) adds, the nodes in matrix $x[r]$, from the $r^{th}$ row to
  the $l^{th}$ row of the $ST$ matrix.

\item \index{function!\textsc{Store}($ST, B$)} The function \textsc{Store}($ST, B$) stores the $ST$ in $B$.

\item \index{function!\textsc{GetSTMinEccentricity}($B$)}The function
  \textsc{GetSTMinEccentricity}($B$), returns the $ST$ that fulfills
  the requirement with minimum increase of eccentricity among all the
  $ST$s stored in $B$. If no $ST$ stored in $B$ fulfills the
  requirement, then the $ST$ that adds maximum number of nodes is
  returned.

\item \index{function!\textsc{GetNodesAdded}($ST, y$)} The function
  \textsc{GetNodesAdded}($ST, y$) returns the number of nodes added in
  $y^{th}$ row of $ST$.

\item \index{function!\textsc{GetSTRow}($v$)}The function
  \textsc{GetSTRow}($v$) returns the index of the row of $ST$ which
  has $v$ node.

\item \index{function!\textsc{GetMinSTRow}($ST$)}The function
  \textsc{GetMinSTRow}($ST$) returns the index of row of $ST$ which
  has minimum number of nodes.

\item \index{function!\textsc{ExtractMinNodes}($ST$)}The function
  \textsc{ExtractMinNodes}($ST$) extracts the minimum number of nodes
  among all the rows of $ST$.

\item \index{function!\textsc{GetEndNodes}($C$)}The function
  \textsc{GetEndNodes}($C$) returns the nodes that are end nodes in
  subtree $C$.

\item \index{function!\textsc{GetSubset}($L(i)$,$j$)}The function
  \textsc{GetSubset}($L(i)$,$j$) returns one of the possible subsets
  of order $j$ from the $i^{th}$ row of matrix $L$.
\end{itemize}

\subsection{Algorithms}

Firstly, we present the algorithms used by the main \textsc{CF}
algorithm. These algorithms build the $m$ subtrees ($ST$ matrix),
extend subtree ($ST$ matrix row) and prune the $m$ subtrees to get
subtrees of order $k$, respectively.

\subsubsection{Algorithm $2$: Building the $ST$ Matrix}

In the \textsc{BuildST} algorithm (Algorithm \ref{ST}), we build the
matrix $ST$ as defined in Section \ref{not}. The algorithm first gets
minimum distance of a node $r$ and its nearest node in
$m$-center. This distance is then compared with the distances of node
$r$ and other nodes in $V_m$. The indices of all the nodes in $V_m$
which are at minimum distance from node $r$ are stored. If there is
only one index stored for node $r$, then the node is added to the
corresponding row of the $ST$. This is repeated for all the nodes
which are in $V(T)$ but not in $V_m$. Then, at the end, for all those
nodes which are at the same distance from more than $1$ nodes in
$V_m$, the $ST$ row with minimum number of nodes among the indices
stored for each node is chosen, and node is added to the selected
row. If there are two or more $ST$ rows with minimum number of nodes,
then the node is added to $ST$ row with lower index value.

In pseudocode for Algorithm \ref{ST}, variable $min$ takes the
distance between node $H[x]$ and the nearest node in $m$-center.
Lines $7$--$9$ check, if $r$ is at equal distance from two or more
node of the $m$-center in $V_m$. If $\mathbf{TRUE}$ then array $u[r]$
holds the indices of the corresponding nodes of the $m$-center,
otherwise array $u[r]$ holds just the index of the nearest node in
$m$-center.

In lines $10$--$14$, if $u[r]$ has only one index, the node $r$ is
added to the $ST$ with index stored in $u[r]$, and $-1$ is added to
$u[r]$ to indicate that this node has been added to some row of
$ST$. The `while' loop in line $4$ repeats this for all nodes in $H$.

\incmargin{1em}
\restylealgo{boxruled}
\linesnumbered
\begin{algorithm}[]       \label{ST}    \index{\textsc{BuildST} Algorithm}
\dontprintsemicolon
\lFor {$i$ $\leftarrow$ $0$ \KwTo $m-1$}{ 
$ST[i][0]$ $\leftarrow$ $V_m[i]$ \; 
}
$H$ $\leftarrow$ $V(T) - V_m$ \; 
$x$ $\leftarrow$ 0 \;  
\While{$H[x]$ $\neq$ $\emptyset$}{ 
$min$ $\leftarrow$ \textsc{ExtractMin}($H[x],V_m$) \;
$h$ $\leftarrow$ $0$ \;
\For {$i$ $\leftarrow$ $0$ \KwTo $m-1$}{
\If {$min$ == $d(H[x],V_m[i])$}{   
$u[x][h++]$ $\leftarrow$ $i$\; 
}
\If{$h == 1$}{
$v$ $\leftarrow$ $0$ \;
\lWhile{$ST[i][v]$ $\neq$ $\emptyset$}{
$v \leftarrow v+1$ \;
}
$ST[i][v]$ $\leftarrow$ $H[x]$ \; 
$u[x][h-1]$ $\leftarrow$ $-1$ \;
}
}
$x \leftarrow x+1$ \;
}
$x$ $\leftarrow$ $0$ \;
\While{$H[x]$ $\neq$ $\emptyset$}{
\If{$u[x][0] \neq$ $-1$}{
$minNum$ $\leftarrow$ \textsc{ExtractMinNum}($u[x]$,$ST$) \;
$minNumRow$ $\leftarrow$ \textsc{ExtractMinNumRow}($u[x],ST$) \; 
$ST[minNumRow][minNum]$ $\leftarrow$ $H[x]$ \;
}
$x \leftarrow x+1$;
}
\Return($ST$) \;
\caption{\textsc{BuildST}($T$, $V_m$)}
\end{algorithm}

At line $18$, `if' condition evaluates $\mathbf{TRUE}$ for those nodes
which are at the same distance from more than one node of the
$m$-center. Line $21$ adds any such node $r$ to the $ST$ row with
minimum number of nodes among the $ST$ rows stored in row $u[r]$. If
two nodes of the $m$-center have equal numbers of nodes, then node $r$
is added to the $ST$ row with lower index value.

Line $23$ returns the $ST$ matrix in which each row contains the nodes
that are assigned to the first node (node in $m$-center) of the
corresponding row.
 
\subsubsection{Algorithm $3$: Extending a Row of the $ST$ Matrix} 

The \textsc{ExtendST} algorithm (Algorithm~\ref{extendST}) extends one
or more rows of $ST$ matrix by taking nodes from other rows of $ST$
matrix. A row $l$ can take nodes from its neighboring row i.e. a row
which has the other end of an edge whose one end is in the row
$l$. The Algorithm \ref{extendST} takes one by one all the nodes in
the $l^{th}$ row of $ST$, which we want to extend, and checks the
adjacent nodes of each node. If the adjacent node is in another row
then we explore that neighbor row. If the neighbor row has sufficient
number of nodes, then we take nodes from that row and add to row
$l$. Otherwise, the neighbor row itself can take nodes from its
neighbor row and serve the requirement of row $l$. When the row $l$
has required number of nodes then algorithm stores the current $ST$
and mark the neighbors that are visited during this formation of
$ST$. Algorithm again starts and in the same way, all the possible
neighbors are explored and corresponding $ST$s are stored. In the end,
the algorithm returns the $ST$ that fulfills the requirement with
minimum increase of eccentricity among all the $ST$s stored. If no
$ST$ fulfills the requirement, then $ST$ that adds maximum number of
nodes is returned. Also the number of nodes added in row $l$ is
returned.

The current $ST$ matrix, $l$, $y$ (index of $ST$ matrix row to
extend), the value of $k$ and $OldST$ matrix are passed as input to
this algorithm. Here, for the initial call of the algorithm $ST$ and
$OldST$ are same. Also values of $l$ and $y$ are exactly same. It
returns the new $ST$ and number of nodes added to row $y$ of $ST$.

In line $1$ of Algorithm~\ref{extendST}, $z[l]$ takes the number of
nodes in $ST[l]$. The `foreach' loop in line $2$ runs for every node
$u$ in $l^{th}$ row of $ST$, i.e. $ST[l]$. Second `foreach' loop
checks all the adjacent nodes of $u$.

If an adjacent node $v$ is not a part of $ST[l]$ and $color[v]$ is not
'$Gray$', then $v$ belongs to one of neighbors of $ST[l]$ that is not
explored yet. From lines $5$--$23$, the neighbor is explored to see if
we can take nodes from this neighbor.

If node $v$ is in $V_m$, then this neighbor of $ST[l]$ cannot give any
nodes, so we continue the `foreach' loop of line $3$ with next
adjacent node of $u$.


At line $5$, $stop$ is set to $\mathbf{FALSE}$ to indicate that some
neighbor is explored in this loop. At line $10$, $d[r]$ is compared
with $k$, if $d[r] \geq k$, then \textsc{AddRemoveNodes} function,
adds nodes in $x[r]$ to $ST[l]$ and removes from $ST[r]$. At line
$12$, the $color[v]$ is set to `$Gray$'.

At line $13$, if number of nodes added in $ST[l]$ minus number of
nodes $ST[l]$ has to give, is greater than $k$, then $l$ is compared
with $y$ to check whether it is the initial call or a recursive
call. If it is the initial call to the algorithm, then control comes
to line $25$. Otherwise, Algorithm \ref{extendST} returns the number
of nodes added in $ST[l]$. If the condition at line $13$ fails, then
it continues the `foreach' loop in line $3$ with other adjacent nodes
of node $u$.

If the condition at line $10$ fails, i.e. $d[r] < k$, then we need to
check if $ST[r]$ can be extended. In this case, we take some nodes
from neighbors of $ST[r]$ and give some nodes to $ST[l]$ in order to
fulfill the need of $ST[l]$. So basically we now extend $ST$ with
changed row number of $ST$, hence a recursive call to Algorithm
\ref{extendST} is made as \textsc{ExtendST}($ST,r,k,y,OldST$).

The value it returns is taken in $n[r]$ and added to $d[r]$, and it
checks whether $d[r]+n[r] \geq k$. If $\mathbf{TRUE}$ it adds the
$x[r]$ nodes to $ST[l]$ and removes then from $ST[r]$. Next, we check
$ST[l]$, if number of nodes in $ST[l]$ minus $x[l]$ is greater than
$k$ then $l$ is compared with $y$ to check whether it is the initial
call or a recursive call. If it is the initial call to algorithm, then
control comes to line $25$. Otherwise, Algorithm \ref{extendST}
returns the number of nodes added in $ST[l]$. But, if the condition at
line $20$ fails, then it continues the loop in line $3$ with other
adjacent nodes of node $u$. If the condition at line $18$ fails, then
$color[v]$ is set to $Gray$.

\incmargin{1em}
\restylealgo{boxruled}
\linesnumbered
\begin{algorithm}[]              \index{\textsc{ExtendST} Algorithm}      
\label{extendST}
\dontprintsemicolon
$z[l] \leftarrow |ST[l]|$ \;
\ForEach {$u \in ST[l]$}{
\ForEach {$v \in adj[u]$}{
\If{$v \notin ST[l]$ and $color[v] \neq Gray$}{
$stop \leftarrow \mathbf{FALSE}$ \;
\lIf{$v \in V_m$}{
continue  \;
}
$r \leftarrow$ \textsc{GetSTRow}($v$) \;
$x[r] \leftarrow$ \textsc{NodesToBeRemoved}($v$) \;
$d[r] \leftarrow |ST[r]|-|x[r]|$ \;
\If{$d[r] \geq k$}{
\textsc{AddRemoveNodes}($x[r]$, $ST[l]$, $ST[r]$) \;
$color[v] \leftarrow Gray$ \;
\If{$|ST[l]|-|x[l]| \geq k$}{
\lIf{$l==y$}{
Goto Line $25$ \;
}
\Return($|ST[l]|-z[l]$) \;
}
\lElse{
Continue \;
}
}
$n[r] \leftarrow$ \textsc{ExtendST}($ST, r, k, y, OldST$) \;
\If{$d[r]+n[r] \geq k$}{
\textsc{AddRemoveNodes}($x[r]$, $ST[l]$, $ST[r]$) \;
\If{$ST[l]-|x[l]| \geq k$}{
\lIf{$l==y$}{
Goto Line $25$ \;
}
\Return($|ST[l]|-z[l]$) \;
}
}
\lElse{$color[v] \leftarrow Gray$ \;
}

}
}
}
\lIf{$l \neq y$}{\Return($|ST[l]|-z[l]$) \;}
\If {$stop$ == $\mathbf{FALSE}$}{
\textsc{Store}($ST, B$) \;
$ST \leftarrow OldST$ \;
$stop \leftarrow \mathbf{TRUE}$ \;
Goto Line $1$ \;
}
$ST$ $\leftarrow$ \textsc{GetSTMinEccentricity}($B$) \;
$nodesAdded$ $\leftarrow$ \textsc{GetNodesAdded}($ST, y$) \;
\Return($nodesAdded$) \;
  
\caption{\textsc{ExtendST}($ST$, $l$, $k$, $y$, $OldST$)}
\end{algorithm}
\decmargin{1em}


If all the nodes in $ST[l]$ have been checked and $ST[l]$ minus $x[l]$
is still less than $k$, then at line $24$, $l$ is compared with $y$ to
check whether it is the initial call or a recursive call. If it is a
recursive call, the number of nodes added to the $l^{th}$ row of $ST$
is returned. If it is a initial call to algorithm, then control comes
to line $25$.

At line $25$, if $stop$ is $\mathbf{FALSE}$, then function Store($ST,
B$) stores the current $ST$ in $B$, current $ST$ is replaced with
$OldST$ and $stop$ is set to $\mathbf{FALSE}$. The algorithm starts
with the initial call parameters to Algorithm \ref{extendST} and this
call generates a new $ST$. In this way we store all the possible $ST$s
in $B$. When no neighbor is there to get nodes, the $ST$ which
fulfills the required nodes with minimum possible increase of
eccentricity is given as final output of the algorithm. This final
$ST$ and number of nodes added in row $y$ are returned.

\subsubsection{Algorithm $4$: Pruning the $ST$ Matrix Rows to Get a Central Forest} 

McMorris and Reid~\cite{McMorrisReid1997} have given a pruning
algorithm for constructing central-$k$-tree in a given tree. We have
to find $m$ central-$k$-subtrees for central forest $C(T;m,k)$. We
have already divided the nodes into $m$ subtrees in $ST$. If we apply
the central-$k$-tree algorithm to each subtree, we can get the
central-$k$-subtree for each subtree.

The \textsc{PruneST} algorithm (Algorithm \ref{pruneST}) prunes
(removes) all the end-nodes of the tree and their corresponding edges
in each step. After pruning, if the nodes in the resulting tree $T'$
are less than or equal to $k$, then a subset of nodes from the nodes
that are pruned in the last step is added to the nodes left in the
tree $T'$ and algorithm exits. The cardinality of this subset is
$k-V(T')$. Otherwise, if the number of nodes in the resulting tree is
more than $k$, then prune all the end-nodes again. And this repeats
for some definite number of steps.

\incmargin{1em}
\restylealgo{boxruled}
\linesnumbered
\begin{algorithm}[]               \index{\textsc{PruneST} Algorithm}      
\label{pruneST}
\dontprintsemicolon
\ForEach{row $r$ of $ST$}{
$C \leftarrow ST[r]$ \;
$C(0) \leftarrow C$ \;
$L(0) \leftarrow$ \textsc{GetEndNodes}($C$) \;
$i \leftarrow 1$ \;
\While{$\mathbf{TRUE}$}{
$C(i) \leftarrow C(i-1)/L(i-1)$ \;
\If{$|V(C(i))| \leq k$}{
$X \leftarrow$ \textsc{GetSubset}($L(i-1)$,$(k-|V(C(i))|)$)  \;
$U[r] \leftarrow V(C(i)) \cup X$ \;
break \;
}
$L(i) \leftarrow$ \textsc{GetEndNodes}($C$) \;
$i \leftarrow i+1$ \;
}
}
\Return ($U$)
\caption{\textsc{PruneST}($ST$, $k$)}
\end{algorithm}
\decmargin{1em}


The Algorithm~\ref{pruneST} is taken from McMorris and Reid
~\cite{McMorrisReid1997} with some slight modifications. In Algorithm
\ref{pruneST}, every time line $2$ sets $C$ as the one of the
subtrees(rows) of $ST$. $L(0)$ gets the end nodes of $C$. As in
McMorris and Reid ~\cite{McMorrisReid1997}, for a subset $L$ of the
nodes of $C$, let $C/L$ denote the sub forest with node set $V(C)/L$
and edge set containing of all edges of $C$ incident with no node in
$L$ i.e now $C/L$ contains all the nodes that are in $V(C)$ but not in
$L$.

At line $8$, if the `if' condition evaluates to $\mathbf{TRUE}$, $X$
takes any subset of $L(i-1)$ of order $k-|V(C(i)|$. The $U[r]$ ($U$
for $ST$ row $r$) is set to $V(C(i))$ $\cup$ $X$ and control comes out
of infinite `while' loop. Then, the `foreach' loop of line $1$ starts
with next row in $ST$.

But if the `if' condition at line $8$ fails, then it continues in the
`while' loop. When all rows of $ST$ are done, the Algorithm
\ref{pruneST} returns the matrix $U$.

\begin{remark}

  The forest \(C(T; m, k)\) does not necessarily have unique subtrees,
  i.e. a $k$ node sub tree can be built by choosing a different set of
  nodes.

\end{remark}

The justification for this remark is that central-$k$-tree of a tree
is not unique. As shown in Section \ref{centralktree}, we may get a
different central-$k$-subtree using the same algorithm. Therefore,
subtrees of $C(T;m,k)$ are also not unique.

\subsubsection{Algorithm $5$: The Main Algorithm for Building a Central Forest}

The \textsc{CF} algorithm (Algorithm \ref{algoCF}) builds a central
forest. The matrix $V_m$, tree $T$ and order of each subtree of
central forest, $k$ are given as input to the algorithm. The value of
$k$ should be less than or equal to the upper bound on $k$ as given by
Remark ~\ref{remark1}. The output of this algorithm is the central
forest with $m$ subtrees of order $k$ each. If central forest of order
$k$ is not possible, then the algorithm outputs the maximum value of
$k$ for which a central forest is possible.

\incmargin{1em}
\restylealgo{boxruled}
\linesnumbered
\begin{algorithm}[]          \index{\textsc{CF} Algorithm}      
\label{algoCF}
\dontprintsemicolon
$ST$ $\leftarrow$ \textsc{BuildST}($T$,$Vset_m[i]$)  \;
$minNodes \leftarrow$  \textsc{ExtractMinNodes}(ST) \;
\While{$k>minNodes$}{
\ForEach{$u \in V(T)$}{
$color[u] \leftarrow White$ \;}
$l \leftarrow$ \textsc{GetMinSTRow}() \;
$OldST \leftarrow ST$ \; 
$a \leftarrow$ \textsc{ExtendST}($ST, l, k, l, OldST$) \;
\If{$a == 0$}{
$ST \leftarrow OldST$ \;
Output $minNodes$  \;
Exit \;
}
$minNodes \leftarrow$  \textsc{ExtractMinNodes}(ST) \;
}
$A \leftarrow$ \textsc{PruneST}($ST$, $k$) \; 
\caption{\textsc{CF}($T$, $V_m$, $k$)}
\end{algorithm}
\decmargin{1em}

Line $1$ calls the Algorithm \ref{ST} to build the matrix $ST$.

The $minNodes$ variable takes the minimum number of nodes in any row
of matrix $ST$. At line $3$, the value of $minNodes$ is compared with
the value of $k$. In the first case, when $k$ is less than or equal to
$minNodes$, the condition at line $3$ fails and Algorithm
\ref{pruneST} is called with parameters $ST$ and $k$ which returns the
central forest.

But in the second case, when $k$ is greater than the $minNodes$, we
need to extend that subtree (row) of $ST$, such that the number of
nodes becomes equal or greater than the $k$. Also, there may be more
than one subtrees (rows) that have nodes less than $k$, so at line $3$
`while' loop repeats lines $4$--$13$ till any of the subtree (row) has
less number of nodes than $k$.

Line $5$ sets the $color$ of each node in $V(T)$ as `$White$'. At line
$8$, a call is made to Algorithm \ref{extendST}. The Algorithm
\ref{extendST} returns the number of nodes added to $l^{th}$ row of
$ST$ and the new $ST$.

If the value of $a$ is $0$, no node can be added to $l^{th}$ row, thus
the central forest for this value of $k$ is not possible. Line $10$
replaces new $ST$ with the $OldST$, i.e. $ST$ before the call to
Algorithm \ref{extendST}. The value of $minNodes$ which is the maximum
value for which the central forest is possible, is given as output and
the Algorithm \ref{algoCF} exits.  As mentioned in Section
\ref{upbound}, this value is the required upper bound on the value of
$k$ for tree $T$.

But if at line $8$, the value returned in $a$ is non-zero, then at
line $13$ the function \textsc{ExtractMinNodes}(ST) extracts the
minimum number of nodes among all the rows of new $ST$. The `while'
loop of line $3$ repeats till the value of $minNodes$ becomes equal to
or less than $k$. When the condition at line $3$ fails, Algorithm
\ref{pruneST} is called that returns a central forest in $A$ matrix.

\subsection{An Example Demonstrating the \textsc{CF} Algorithm}

We take a short example as shown in Figure \ref{ExTree} to demonstrate
the working of Algorithm \ref{algoCF}. In this example, we are
constructing central forests stated as follows:

\begin{figure}[htp]                \index{example tree}
\begin{center}
\includegraphics [scale=.5, angle = -90] {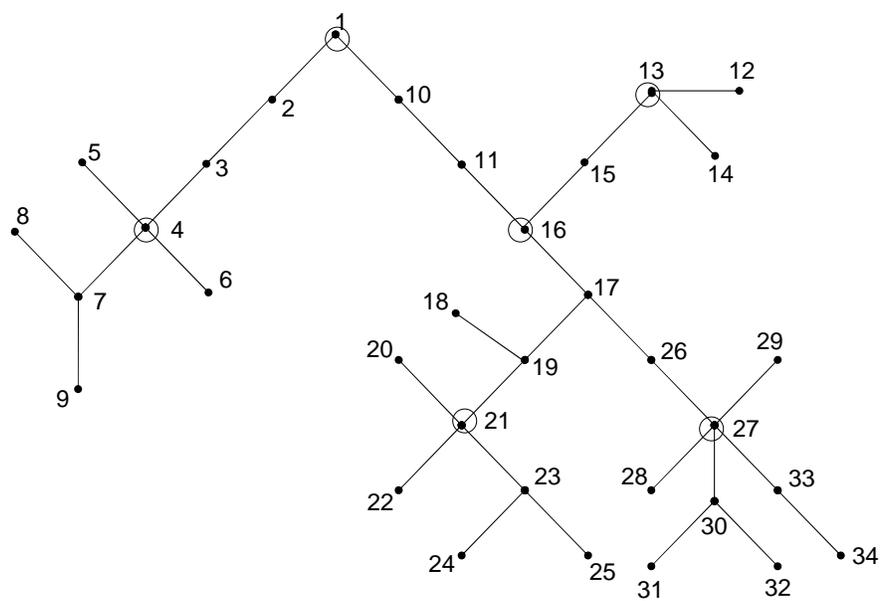}
\caption{Example tree $T$}
\label{ExTree}
\end{center}
\end{figure} 

\begin{itemize}
\item [1.] $C(T;6,2)$ and  
\item [2.] $C(T;6,4)$
\end{itemize}

We are given\\
$V(T) = \{1,2,3,4,5,6,7,8,9,10,11,12,13,14,15,16,17,18,19,20,21,22,$ \\
$23,24,25,26,27,28,29,30,31,32,33,34\}$ \\
$V_6 = \{1,4,13,16,21,27\}$ \\

\begin{itemize}
\item [1.] For $C(T;6,2)$ \\  

  Algorithm~\ref{algoCF} calls Algorithm~\ref{ST} to build the $ST$
  matrix. In Algorithm~\ref{ST}, $H$ =
  $\{2,3,5,6,7,8,9,10,11,12,14,15,17,18,19,20,22,23,24,25,$\\
  $26,28,29,30,31,$ $32,33,34\}$. Except node $15$, all the nodes have
  a unique nearest node in $V_6$. So all these nodes are directly
  added to the corresponding rows in $ST$ matrix and the $ST$ matrix
  is as shown in Table~\ref{ST1}.

  At the end we check the array $u$ for non-zero values. Only the $u$
  array of node $15$ holds the index $3$ and $4$. The number of nodes
  in row $3$ is equal to the number of nodes in row $4$, so node $15$
  is added to the row with lower index value in $V_6$, i.e. row
  $3$. The final $ST$ matrix, as depicted in Table~\ref{ST2}, is
  returned to the Algorithm~\ref{algoCF}.

\begin{table}[htp]                
\centering
\begin{tabular}{|l|l|l|l|l|l|l|l|l|l|l} 
\hline
\backslashbox{$row$}{$column$}&1&2&3&4&5&6&7&8&9\\
\hline
1&1&2&10&&&&&&\\
\hline
2&4&3&5&6&7&8&9&&\\
\hline
3&13&12&14&&&&&&\\
\hline
4&16&11&17&&&&&&\\
\hline
5&21&18&19&20&22&23&24&25&\\
\hline
6&27&26&28&29&30&31&32&33&34\\
\hline
\end{tabular} 
\caption{Intermediate $ST$ matrix}
\label{ST1}
\end{table}

\begin{table}[htp]                 
\centering
\begin{tabular}{|l|l|l|l|l|l|l|l|l|l|l} 
\hline
\backslashbox{$row$}{$column$}&1&2&3&4&5&6&7&8&9\\
\hline
1&1&2&10&&&&&&\\
\hline
2&4&3&5&6&7&8&9&&\\
\hline
3&13&12&14&15&&&&&\\
\hline
4&16&11&17&&&&&&\\
\hline
5&21&18&19&20&22&23&24&25&\\
\hline
6&27&26&28&29&30&31&32&33&34\\
\hline
\end{tabular} 
\caption{$ST$ matrix}
\label{ST2}
\end{table}

Line $2$ of Algorithm~\ref{algoCF} sets $minNodes = 3$. The condition
at line $3$ fails as $k=2$ here. So at line $14$
Algorithm~\ref{pruneST} is called.

In Algorithm~\ref{pruneST}, for the first time let $r$ be $1$ and $C$
gets $\{1,2,10\}$,$C(0) = \{1,2,10\}$ $L(0)= \{2,10\},i=1$. The
condition at line $6$ is always true, so line $7$ sets $C(1) =
C(0)/L(0) = \{1\}$. At line $8$, $|V(C(1)| = 1$, is less than $k$ so
$U[1] = \{1\} \cup X$, where $X$ can be $\{2\}$ or $\{10\}$. Let $X=
\{2\}$, so $U[1] = \{1,2\}$. Similarly $U$ is calculated for every row
of $ST$. Then the matrix $U$ is returned to Algorithm~\ref{algoCF}.

At line $14$ of Algorithm~\ref{algoCF}, when the call to
Algorithm~\ref{pruneST} returns, $A$ holds the returned central
forest.

The central forest is shown in Figure~\ref{CF62ExTree}. Here ovals are
used to represent the nodes in each subtree of central forest. There
are total $6$ ovals, each of which has $2$ nodes, thus we get
$C(T;6,2)$.

\begin{figure}[htp]                   \index{$C(T;6,2)$}
\begin{center}
\includegraphics [scale=.5, angle = -90] {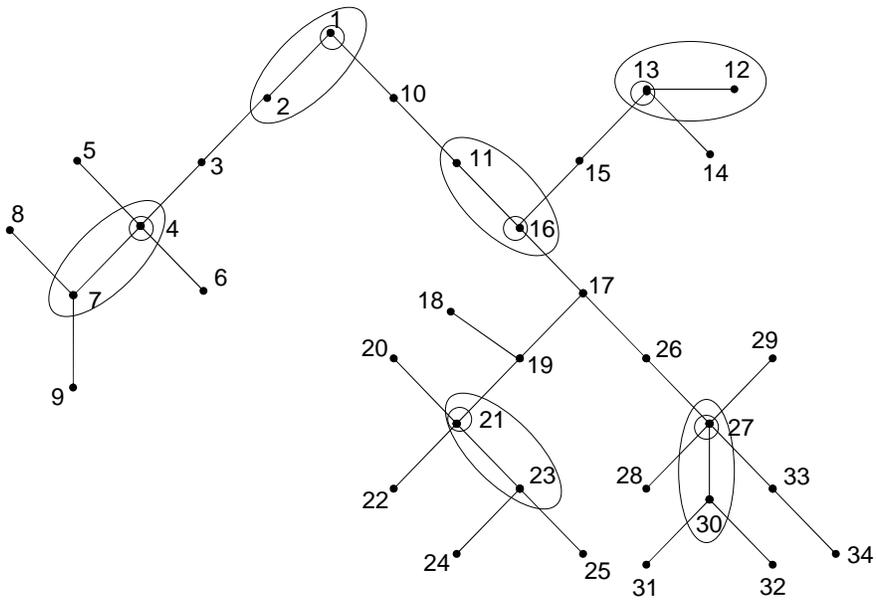}
\caption{A central forest $C(T;6,2)$ of example tree $T$}
\label{CF62ExTree}
\end{center}
\end{figure}

\item [2.] For $C(T;6,4)$ \\

  Line $1$ of Algorithm \ref{algoCF} builds $ST$ matrix in the same
  way as in case $1$.

  Line $2$ of Algorithm \ref{algoCF} sets $minNodes = 3$. This time
  the condition at line $3$ is true. Lines $4$--$5$ set the $color$ of
  each node of $V(T)$ as `$White$'. Lines $6$ and $7$ set $l = 1$ and
  $OldST = ST$. At line $8$, \textsc{ExtendST}($ST,1,4,1,OldST$)
  Algorithm \ref{extendST} is called.

  In Algorithm \ref{extendST}, $z[1] = 3$ as $ST[1] = \{1,2,10\}$. The
  `for' loop in line $2$ starts with node $1$, `for' loop in line $3$
  checks nodes $2$ and $10$ (adjacent of $1$) but `if' condition fails
  at line $4$. So now `for' loop at line $2$ starts with next node in
  $ST[l]$ which is node $2$. Its adjacent node $3$ does not belong to
  $ST[1]$ and $color[3]$ is not `$Gray$', so $stop$ is set to
  $\mathbf{FALSE}$. Also node $3$ does not belong to $V_6$ so lines
  $7-9$ set $r = 2, x[2] = \{3\}, d[2]= 6$.

  The `if' condition at line $10$ is true and node $3$ is added to
  $ST[1]$ and removed from $ST[2]$. The variable $color[3]$ is set to
  `$Gray$'. Condition at line $13$ is also true, so control comes to
  line $25$. Condition at line $25$ is $\mathbf{TRUE}$, so current
  $ST$ is stored in matrix $B$, current $ST$ is replaced with $OldST$,
  $stop$ is set to $\mathbf{TRUE}$ and control is transfered to line
  $1$.

  At line $1$ of Algorithm~\ref{extendST}, $z[1] = 3$ again as we have
  $ST[1] = \{1,2,10\}$. Proceeding in the same way, when the `for'
  loop at line $2$ starts with node $2$. Its adjacent node $3$ does
  not belong to $ST[1]$ but $color[3]$ is `$Gray$', so `for' loop of
  line $2$ starts with next node in $ST[1]$, which is node $10$. Its
  adjacent node $11$ does not belong to $ST[1]$ and $color[11]$ is not
  `$Gray$', so $stop$ is set to $\mathbf{FALSE}$. Also node $11$ does
  not belong to $V_6$ so lines $7-9$ set $r = 4, x[4] = \{11\}, d[4]=
  2$.
   
  Now condition at line $10$ fails and control comes to line $17$
  where a recursive call is made to Algorithm~\ref{extendST} as
  \textsc{ExtendST}($ST,4,4,1,OldST$). This call starts the algorithm
  with $l=4$ and set $z[4] = 3$. The `for' loop in line $2$ checks all
  the nodes in $ST[4]$ one by one and for node $16$, `if' condition of
  line $4$ evaluates $\mathbf{TRUE}$ as its adjacent node $15$ is
  neither in $ST[4]$ and nor `$Gray$'.

  So, $stop$ is set to $\mathbf{FALSE}$. Node $15$ also does not
  belong to $V_6$ so lines $7-9$ set $r = 3, x[3] = \{15\}, d[3] = 3$.

  Now condition at line $10$ fails and control comes to line $17$
  where a recursive call is made to Algorithm~\ref{extendST} as
  \textsc{ExtendST}($ST,3,4,1,OldST$). This call returns $0$ as
  $ST[3]$ does not have enough nodes to give and also it does not have
  any other neighbor.

  The condition at line $18$ fails as $n[3]$ is $0$ and $d[3]$ = $3$,
  so in else part $color[15]$ is set to `$Gray$'.

  The `foreach' loop of line $3$ continues with other adjacent nodes
  of node $16$.  There is no such node. So loop at line $2$ continues
  with next node in $ST[4]$, which is node $17$.

  Node $19$ is adjacent to node $17$ and it satisfy condition at line
  $4$ so $stop$ is set to $\mathbf{FALSE}$. The condition at line $6$
  fails so lines $7$--$9$ set $r = 5, x[5] = \{18,19\}, d[5] = 6$

  Now condition at line $10$ is true and nodes $18$ and $19$ are added
  to $ST[4]$ and removed from $ST[5]$. The variable $color[19]$ is set
  to `$Gray$'. Condition at line $13$ is $\mathbf{TRUE}$ but here $l$
  is 4 which is not equal to $y$, so nodes added in $ST[4]$ is
  returned and $n[4]$ gets $2$.

\begin{figure}[htp]                 \index{$C(T;6,4)$}
\begin{center}
\includegraphics [scale=.5, angle = -90] {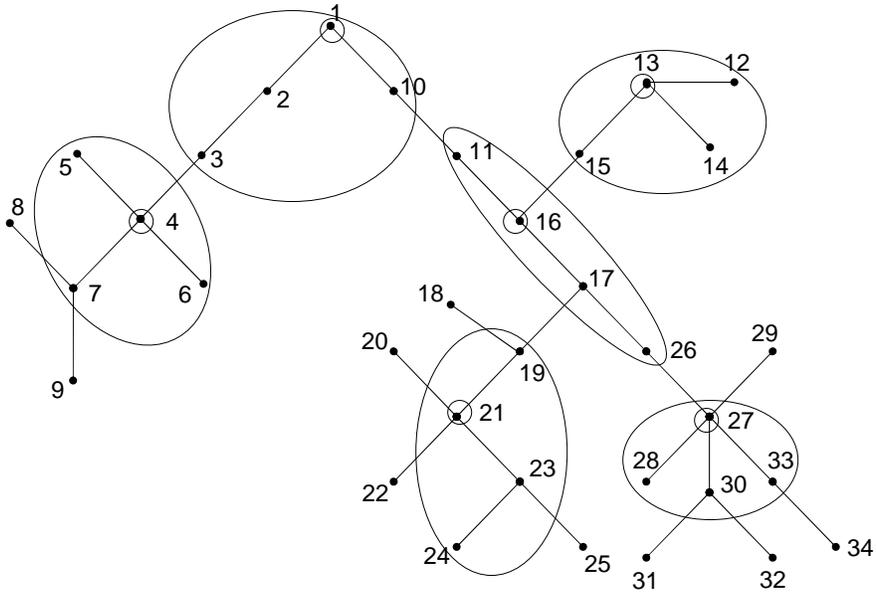}
\caption{A central forest $C(T;6,4)$ of example tree $T$}
\label{CF64ExTree}
\end{center}
\end{figure} 

At line $18$, $n[4]+d[4]$ = 4 is equal $k$ so node $11$ is added to
$ST[1]$ and removed from $ST[4]$ and as $l = y$ the control is
transferred to line $25$. The lines $26$--$29$, store current $ST$ in
matrix $B$, replace current $ST$ with $OldST$, set $stop$ to
$\mathbf{TRUE}$ and transfer to line $1$.

Proceeding in the same way, we get $z[1]$ = $3$, for $u = 10$ and $v =
11$, $stop$ is set to $\mathbf{FALSE}$ and $r = 4, x[4] = \{11\}, d[4]
= 2$. As $d[4]$ is less than $k$ so a recursive call is made to
Algorithm \ref{extendST} as \textsc{ExtendST}($ST,4,4,1,OldST$).

This call starts the algorithm with $l=4$ and set $z[4] = 3$. For $u$
= $17$ and $v = 26$, conditions at lines $4$ and $6$ evaluate
$\mathbf{TRUE}$, so $stop$ is set to $\mathbf{FALSE}$ and $r = 6, x[6]
= \{26\}, d[6] = 8$.

Now condition at line $10$ is true and node $26$ is added to $ST[4]$
and removed from $ST[6]$. $color[26]$ is set to `$Gray$'. Condition at
line $13$ fails, the `foreach' loop in line $3$ continues with other
adjacent nodes of node $17$. But as there is no such node, the
`foreach' loop of $2$ checks next node in ST[4], but conditions at
line $4$ and $6$ evaluate $\mathbf{FALSE}$ for every node. At line
$24$, `if' condition is true as $l = 4$, so $color[11]$ is set to
`$Gray$'. The lines $26$--$29$, store current $ST$ in matrix $B$,
replace current $ST$ with $OldST$, set $stop$ to $\mathbf{TRUE}$ and
transfer to line $1$.

Here, $z[1]$ = $3$, condition at line $4$ evaluates $\mathbf{FALSE}$
for every node in $ST[1]$. At line $25$, `if' condition fails as
$stop$ is $\mathbf{TRUE}$. At line $30$, $ST$ gets the $ST$ that
fulfills the requirement with minimum increase of eccentricity among
all the $ST$s stored in $B$. If no $ST$ stored in $B$ that fulfills
the requirement, then $ST$ that adds maximum number of nodes is
returned. In our case, we choose $ST$ that adds node $3$ in $ST[1]$
from $ST[2]$. The $nodesAdded$ variable gets $1$ as the nodes added in
row $1$ of $ST$ and returns to the main Algorithm~\ref{algoCF}.

In Algorithm~\ref{algoCF}, condition at line $9$ fails, as $1$ node is
added to $ST[l]$ and line $13$ sets $minNodes = 3$. The `while' loop
at line $3$ evaluates $\mathbf{TRUE}$. Lines $4$--$5$ set the $color$
of each node of $V(T)$ as `$White$'. Line $11$ calls
\textsc{ExtendST}($ST,4,4,4,OldST$). Proceeding in the same way as
before the call to Algorithm~\ref{extendST} returns new $ST$ with node
$26$ added to $ST[4]$ and removed from $ST[6]$.

\begin{table}[htp]                 
\centering
\begin{tabular}{|l|l|l|l|l|l|l|l|l|l|l} 
\hline
\backslashbox{$row$}{$column$}&1&2&3&4&5&6&7&8&9\\
\hline
1&1&2&10&3&&&&&\\
\hline
2&4&3&5&6&7&8&9&&\\
\hline
3&13&12&14&15&&&&&\\
\hline
4&16&11&17&26&&&&&\\
\hline
5&21&18&19&20&22&23&24&25&\\
\hline
6&27&26&28&29&30&31&32&33&34\\
\hline
\end{tabular} 
\caption{Final $ST$ matrix}
\label{ST3}
\end{table}

The condition at line $9$ fails again and $minNodes$ gets $4$ which is
equal to $k$, so condition at line $3$ fails. The final $ST$ matrix we
got is shown in Figure~\ref{ST3}.

Then, control comes to line $14$ where Algorithm~\ref{pruneST} is
called. The Algorithm~\ref{pruneST} works in the same way as explained
for case $1$ and we get central forest in $A$.

The central forest so obtained is shown in
Figure~\ref{CF64ExTree}. Here ovals are used to represent the nodes in
each subtree of central forest. There are total $6$ ovals, each of
which has $4$ nodes, thus we get $C(T;6,4)$.
\end{itemize}

\subsection{Performance Analysis of the \textsc{CF} Algorithm}

The running time of the \textsc{CF} algorithm depends upon the $ST$
returned by the Algorithm \ref{ST}. If the nodes are assigned almost
equally to all the nodes of the $m$-center then the
Algorithm~\ref{algoCF} runs very fast, but if the assignment of nodes
is unbalanced, then it may take a long time depending on how large a
value of $k$ is required. The time taken by Algorithm~\ref{algoCF} is
the sum of the running time of Algorithm~\ref{ST},
Algorithm~\ref{extendST}, Algorithm \ref{pruneST} and the time taken
by itself. In this section, we investigate how the
Algorithm~\ref{algoCF} performs. In any case, the Algorithm~\ref{ST}
costs $O(n(m+k))$ time, where $n$ is the number of nodes in the tree
$T$. Algorithm~\ref{extendST} visit each node maximum of one or two
times, so this takes $O(n)$ time. Algorithm~\ref{pruneST} also takes
$O(n)$ time.  In Algorithm~\ref{algoCF} other statements takes
constant amount of time.


\begin{itemize}
\item \noindent{\bf Best Case Analysis} \\
\index{performance analysis!best case}

The best case occurs when the required value of $k$ is less than or
equal to the minimum number of nodes in any subtree of $ST$. This case
mostly occurs when the $ST$ returns the subtree of almost the same
order, i.e. the nodes are equally assigned to all the nodes of the
$m$-center.  Then, Algorithm~\ref{extendST} need not be called even
once. So the running time of Algorithm~\ref{algoCF} is as follows:

$T(n)$ = Cost of Algorithm~\ref{ST} + Cost of Algorithm~\ref{pruneST}
+ Extra Cost of Algorithm~\ref{algoCF}

$T(n)$ = $O(n(m+k))$ + $O(n)$ + $c$ \\
$T(n)$ = $O(n(m+k))$
 
\item \noindent{\bf Worst Case Analysis} \\
\index{performance analysis!worst case}

Worst case occurs when $m-1$ subtrees of $ST$ returned by
Algorithm~\ref{ST} have less number of nodes in their subtrees than
value of $k$. This can happen when one of the subtrees, which is in
the center, has most of the nodes and all others are connected to the
one in the center. The Algorithm~\ref{extendST} can be called once for
all the $m-1$ subtrees and if for some row, number of nodes added is
less than required, then Algorithm~\ref{extendST} is called once again
for that row and this time Algorithm~\ref{extendST} returns $0$, so
Algorithm~\ref{algoCF} exits. Thus, the Algorithm~\ref{extendST} can
be called maximum of $m$ times. Each of the call to
Algorithm~\ref{extendST} costs $O(n)$. Therefore, $T(n)$ = Cost of
Algorithm~\ref{ST} + $m$ $\times$ Cost of Algorithm~\ref{extendST} +
Cost of Algorithm~\ref{pruneST} + Extra Cost of Algorithm~\ref{algoCF}

$T(n)$ = $O(n(m+k))$ + $O(mn)$ + $O(n)$ + $c$ \\
$T(n)$ = $O(n(m+k))$

\end{itemize}

As we see, the worst-case behavior of the proposed algorithm is the same
as its best case.

\subsection{Proof of Correctness of the \textsc{CF} Algorithm}

\begin{remark} \label{proofST}    \index{proof of correctness!\textsc{BuildST} algorithm}

  Algorithm~\ref{ST} builds an $ST$ matrix where each node is assigned
  to its nearest $m$-center node.

\end{remark}


  As shown in Minieka~\cite{Edward1970}, an $m$-center set $X$ of a
  graph is any set of $m$ nodes, that minimizes the maximum distance
  from a node to its nearest node in $m$-center. In other words, the
  eccentricity of set $X$ is less than or equal to any possible set of
  cardinality $m$ of a graph. If such a set $X$ is given, each node
  can be arbitrarily assigned to its nearest node from the $m$-center.
  Thus all the nodes have minimum distance from the $m$-center.  Such
  a set $X$ is given as input to Algorithm~\ref{ST} and the algorithm
  assigns the nodes to their nearest nodes in $m$-center.


\begin{lemma} \label{proofextendST}
\index{proof of correctness!\textsc{ExtendST} algorithm}

The re-arrangement of nodes in $ST$ rows by Algorithm~\ref{extendST}
gives the rows in the new $ST$ matrix to minimize the distance from
the $m$-center.


\end{lemma}
\begin{proof}

  All the nodes in the $ST$ matrix are so arranged by
  Algorithm~\ref{ST} that their distances from the $m$-center are
  minimized.  If we try to rearrange the nodes among the neighbors,
  the distances of the nodes may remain same (if a node is at equal
  distance from both the nodes of the $m$-center) or it may increase.
  We add nodes to a row $r$ only when the number of nodes in row $r$
  is less than the value of $k$.  The algorithm tries all the possible
  node(s) from neighbors that can be added to row $r$ and adds the
  node(s) which gives minimum increase of eccentricity for the
  $m$-center.  So the overall eccentricity remains as small as
  possible.  Obviously, there may be some increase in overall
  eccentricity but this cannot be avoided as we want number of nodes
  in row $r$ to be greater than or equal to $k$.  Thus,
  Algorithm~\ref{extendST} gives the rows in new $ST$ with minimum
  increase in eccentricity. \qedhere

\end{proof}

\begin{lemma}  \label{proofpruneST} 
\index{proof of correctness!\textsc{PruneST} algorithm}

Algorithm~\ref{pruneST} gives the $m$ central-$k$-subtrees for $m$
subtrees in $ST$.

\end{lemma}

\begin{proof}

  The nodes of $T$ are arranged in $ST$, each row of $ST$ is viewed as
  a separate subtree. McMorris and Reid~\cite{McMorrisReid1997} have
  given a proof that their algorithm outputs the central $k$-tree for
  any given tree.  If we apply their algorithm for a subtree of $T$
  then it gives us a central-$k$-subtree.  In our
  Algorithm~\ref{pruneST}, we are using McMorris's and Reid's
  algorithm for each row of $ST$. So at the end we get $m$
  central-$k$-subtrees for rows in $ST$. \qedhere

\end{proof}



\begin{lemma}        \label{proofmcenter}

  The $C(T;m,k)$ we get by building subtrees on $m$-center of tree $T$
  is the same as a $C(T;m,k)$ we get by building subtrees on the
  $1$-centers of members of $C(T;m,k)$.

\end{lemma}

\begin{proof}

  We need to build $m$-subtrees in such a way that when we get a
  central-$k$-subtree of these $m$-subtrees, the eccentricity of each
  subtree is minimum. If we build our $m$-subtrees around the
  $1$-center of each member of $C(T;m,k)$ and we prune these subtrees
  to get central-$k$-subtrees, we get each member of $C(T;m,k)$. If we
  build subtrees around any other node from each members of
  $C(T;m,k)$, we get almost same subtrees as we get by choosing
  $1$-centers except some of the nodes may be assigned to the
  neighboring subtrees. By Theorem~\ref{mcenterpart}, the $m$-center
  of tree $T$, are part of members of some $C(T;m,k)$. Therefore, we
  know at least one node from each members of $C(T;m,k)$. We build
  $m$-subtrees around the $m$-center of tree $T$. As some of nodes may
  be assigned to the neighboring subtrees, so if we need more nodes in
  any subtree to build the subtree of order $k$, we take nodes from
  neighboring subtrees. Then, We apply pruning on these $m$-subtrees
  and get central-$k$-subtrees as members of $C(T;m,k)$. \qedhere

\end{proof}

\begin{theorem}  \label{proofalgoCF}
\index{proof of correctness!\textsc{CF} algorithm}

The Algorithm~\ref{algoCF} gives either the central forest $C(T;m,k)$
for the tree $T$ with $m$ subtrees of order $k$ or outputs maximum
possible order of a subtree and exits.

\end{theorem}

\begin{proof}

  The algorithm eventually terminates. Using Remark~\ref{proofST},
  Lemma~\ref{proofextendST}, Lemma~\ref{proofpruneST} and
  Lemma~\ref{proofmcenter}, it can be easily shown that
  Algorithm~\ref{pruneST} outputs the forest with minimum
  eccentricity.  Also, if Algorithm~\ref{extendST} returns $0$, the
  Algorithm~\ref{algoCF} terminates and outputs the maximum possible
  value of $k$ for which central forest is possible. \qedhere

\end{proof}

\section{Conclusions and Further Work}   \label{conc}

In this paper we have introduced a new central structure in trees,
which we call central forests in trees. $C(T;m,k)$ is a central forest
of $m$ subtrees, each of order $k$, for tree $T$, which has minimum
eccentricity among all the possible forests of this order for tree
$T$. We have given an algorithm for constructing the central forests
in trees. This algorithm is efficient as it computes the central
forest in $O(n(m+k))$ time, where $n$ is the number of nodes in the
tree $T$. The complete analysis and proof of correctness for the
algorithm are also given in this paper. Our algorithm also computes a
upper bound on the value of $k$ for which the central forest of $m$
subtrees is possible.

This work suggests the following possible extensions.
\begin{itemize}
\item[(1)] Further Generalization
\begin{itemize}
\item A further \index{generalization} generalization is possible, if
  we allow that not all the subtrees in a central forest may be of
  like order.

Then if $\mathbb{F}_{T,m}$ is the set of all forests in $T$ of $m$
subtrees each, and \(\vec{\mathbf{a}} = (a_0, a_1, \ldots, a_{m-1})\)
is a vector giving the orders of the $m$ subtrees with \(a_i \neq 0\),
then the eccentricity of a single forest is as given before, and a
central forest \(C(T; m, \vec{\mathbf{a}})\) is given by:
\[ C(T; m, \vec{\mathbf{a}}) \ \triangleq \ \{ F \ | \ e_m(F) \leq e_m(F'), \ \forall F' \in \mathbb{F}_{T,m} \}. \]





\item An interesting problem can be the study of forests $\Psi$ of
  subtrees (of variable orders and number) of a tree $T$ when the
  maximal allowable eccentricity $\delta$ is specified.  One can give
  an algorithm that takes maximum allowable eccentricity and outputs
  the value of $m$ and $k$ for a $C(T;m,k)$. The problem is not well
  understood yet, it may be possible that no such algorithm exists and
  the problem posed is NP-hard.

\item Another potential area is to explore a similar extension to the
  concept of centroids, by defining a centroidal forest. Study the
  properties of centroidal forest and give a algorithm to construct
  such a forest in trees.

\item Another possible future work can be centrally use not more than
  $\sigma$ nodes and create not more than $\rho$ subtrees. That is, in
  central forest, there is a limit $\sigma$ on number of vertices that
  can be used in central forest. Also the number of subtrees ($m$)
  cannot be more than $\rho$ for a central forest. Under these
  restrictions, what is the minimum eccentricity possible for a
  central forest in a given tree?

\end{itemize}
\end{itemize}
\begin{itemize}
\item[(2)] Further Work on the CF Algorithm
\begin{itemize}

\item The algorithm we have given for constructing central forest is a
  centralized algorithm. A distributed algorithm, where nodes can
  decide whether they are part of a $C(T;m,k)$ or not, can be created.

\item The upper bound on the value of $k$ is given by the algorithm,
  we do not have a way to express the upper bound in terms of the
  degree of the tree, eccentricity, etc. Our upper bound depends on
  the algorithm output. A more generalized expression for upper bound
  on the value of $k$ for $m$ subtrees may be found.

\item We are assuming that an $m$-center of the tree is given and
  building our algorithm using this. One can think of an algorithm
  that does not require an $m$-center, or calculates this by itself.

\end{itemize}
\end{itemize}

\bibliography{forestsbib}

\end{document}